\documentclass[letterpaper,11pt,openany,oneside]{article}

\usepackage[T1]{fontenc}
\usepackage[latin1]{inputenc}
\usepackage{epsfig}
\usepackage{amsmath,amssymb}
\usepackage[active]{srcltx}
\usepackage{dsfont}
\usepackage{mathtools}
\usepackage{xcolor,graphicx}
\usepackage{tikz}
\numberwithin{equation}{section}
\newtheorem{theoreme}{Theorem}[section]
\newtheorem{proposition}[theoreme]{Proposition}

\newtheorem{lemme}[theoreme]{Lemma}

\newtheorem{definition}[theoreme]{Definition}

\newcommand{\myendproof}{\hfill\strut\nobreak\hfill\tombstone\par\medbreak} 
\newcommand{\tombstone}{\hbox{\lower.4pt\vbox{\hrule\hbox{\vrule
  \kern7.6pt\vrule height7.6pt}\hrule}\kern.5pt}} 
\newenvironment{proof}[1][Proof]{\noindent \textbf{#1.}~ }{\myendproof}

\newcommand{\R}{\ensuremath{\mathbb R}}
\newcommand{\val}{\mathsf{val}}

\newcommand{\E}{\ensuremath{\mathbb E}}

\providecommand{\prob}{\mathsf{P}}

\newcommand{\N}{\ensuremath{\mathbb N}}

\newcommand{\xx}{\ensuremath{\mathbf x}}
\newcommand{\yy}{\ensuremath{\mathbf y}}

\newcommand{\Id}{\ensuremath{\operatorname{Id}}}

\newcommand{\De}{\Delta}

\newcommand{\ii}{\ensuremath{\mathbf i}}
\newcommand{\jj}{\ensuremath{\mathbf j}}

\newcommand{\la}{\lambda}
\newcommand{\ep}{\varepsilon}

\newcommand{\al}{\alpha}
\newcommand{\be}{\beta}
\newcommand{\si}{\sigma}
\newcommand{\bit}{\mathsf{bit}}

\newcommand{\ga}{\gamma}

\definecolor{darkblue}{rgb}{0,0,0.7} 

\title{A formula for the value of a stochastic game}

\author{Luc Attia\footnote{\'Ecole Polytechnique, Palaiseau, France.}$\ $ and Miquel Oliu-Barton\footnote{Universit\'e Paris-Dauphine, PSL, France. Email: miquel.oliu.barton@normalesup.org}}

\date{\today\\[0.4cm]\footnotesize{\emph{Proceedings of the National Academy of Science of the USA (2019), https://doi.org/10.1073/pnas.1908643116}}\\[0.3cm]
Edited by Robert J. Aumann, The Hebrew University, Jerusalem, Israel, and approved November 12, 2019}
\normalsize{}

\begin{document}
 \maketitle
\abstract{In 1953, Lloyd Shapley defined the model of stochastic games, which were the first general dynamic model of a game to be defined, and proved that competitive stochastic games have a discounted value. In 1982, Jean-Fran\c{c}ois Mertens and Abraham Neyman 
proved that competitive stochastic games admit a robust solution concept, the value, which is equal to the limit of the discounted values as the discount rate goes to 0. 
Both contributions were published in PNAS. 
In the present paper, we provide a tractable formula for the value of competitive stochastic games.} 

\noindent \paragraph{Significance Statement.} 
Stochastic games were introduced by the Nobel Memorial Prize winner Lloyd Shapley in 1953 in order to model dynamic interactions in which the environment changes in response to the players' behavior. The theory of stochastic games and its applications have been studied in several scientific disciplines, including economics, operations research, evolutionary biology, and computer science. In addition, mathematical tools that were used and developed in the study of stochastic games are used by mathematicians and computer scientists in other fields. This paper contributes to the theory of stochastic games by providing a tractable formula for the value of finite competitive stochastic games. This result settles a major open problem which remained unsolved for nearly 40 years. 

\section{Introduction}
\subsection{Motivation}
Stochastic games are the first general model of dynamic
games. Introduced by Shapley \cite{shapley53} in 1953, stochastic
games extend the model of strategic-form games, which is due
to von Neumann \cite{VN28}, to dynamic situations in which
the environment (henceforth, the {state}) changes in response
to the players' choices. They also extend the model of
Markov decision problems to competitive situations with more
than one decision-maker. 

Stochastic games proceed in stages. At each stage, the
players choose actions which are available to them at the
current state.
Their choices have two effects: 
they generate a stage {reward} for each player, and
they determine the probability for the state at
the next stage. Consequently, the players are typically confronted
with a trade-off between getting high rewards in the present
and trying to reach states that will ensure high future rewards.
Stochastic games and their applications have
been studied in several scientific disciplines, including
economics, operations research, evolutionary biology, and
computer science. In addition, mathematical tools that were
used and developed in the study of stochastic games are used
by mathematicians and computer scientists in other fields.
We refer the readers to Solan and Vieille \cite{SV15} for a
summary of the historical context and the impact of
Shapley's seminal contribution.

The present paper
deals with finite competitive stochastic games, that is:
two-player stochastic games with finitely many states and
actions, and where the stage rewards of the players add up to zero. 
Shapley \cite{shapley53} proved that these games have a discounted value, which
represents what playing the game is worth to the players when future rewards are discounted at a constant positive rate. 
Bewley and Kohlberg \cite{BK76} proved that the discounted values admit a limit as the discount rate goes to 0.
Building on this result, Mertens and Neyman \cite{MN81, MN82} proved that finite competitive stochastic games admit a robust solution concept, the value, which represents what playing the game is worth to the players when they are sufficiently patient. 

Finding a tractable formula for the value of finite competitive stochastic games was major open problem for nearly 40 years, which is settled in the present contribution.
While opening a new path for faster computations, our approach may also bring new quantitative and qualitative insights on the model of stochastic games.

\subsection{Outline of the paper}
The paper is organized as follows. Section~\ref{prelim} states
our results on finite competitive stochastic games, namely a
formula for the $\la$-discounted values (proved in
Section~\ref{v_la}), and a formula for the value (proved in
Section~\ref{the_limit}).
Section~\ref{algos} describes the algorithmic implications and
tractability of these two formulas.
Section~\ref{comments} concludes with remarks and extensions.

\section{Context and main results} \label{prelim}
In order to state our results precisely, we recall some
definitions and well-known results about two-player zero-sum
games (Section~\ref{NF}) and about finite competitive stochastic
games (Section~\ref{NF2}).
In Section~\ref{stateart} we give a brief overview of the relevant
literature on finite competitive stochastic games.
Our results are described in Section~\ref{main}.

\paragraph{Notation.} Throughout the paper, $\N$ denotes the set of positive integers.
For any finite set $E$ we denote the set of probabilities over $E$
by $\De(E)=\{f:E\to [0,1]\mid \sum_{e\in E}f(e)=1\}$ and its
cardinality by $|E|$.

\subsection{Preliminaries on zero-sum games}\label{NF}
The aim of this section is to recall some well-known
definitions and facts about two-player zero-sum games,
henceforth zero-sum games.

\paragraph{Definition.}
A \textit{zero-sum game} is described by a triplet
$(S,T,\rho)$ where $S$ and $T$ are the sets of possible
strategies for Player 1 and 2, respectively, and
$\rho:S\times T\to \R$ is a pay-off function. It is played
as follows. Independently and simultaneously, the first
player chooses $s\in S$ and the second player chooses $t\in T$.
Player 1 receives $\rho(s,t)$ and Player 2 receives $-\rho(s,t)$. 
The zero-sum game $(S,T,\rho)$ has a \emph{value} whenever
\[
\sup_{s\in S}\inf_{t\in T} \rho(s,t)= \inf_{t\in
T}\sup_{s\in S}\rho(s,t).
\]
In this case, we denote this common quantity by $\val\, \rho$.  

\paragraph{Optimal strategies.} Let $(S,T,\rho)$ be a
zero-sum game which has a value. An \textit{optimal
strategy} for Player 1 is an element $s^*\in S$ so that
$\rho(s^*,t)\geq \val\,\rho$ for all $t\in T$.
Similarly, $t^*\in T$ is an optimal strategy for Player~2 if
$\rho(s,t^*)\leq\val\, \rho$ for all $s\in S$.

\paragraph{The value operator.} The following properties are well-known: 
\begin{itemize}\item[$(i)$] \emph{Minmax theorem}. Let $(S,T,\rho)$ be a zero-sum game. Suppose that $S$ and $T$ are two compact subsets of some topological vector space, $\rho$ is a continuous function, the map $s\mapsto \rho(s,t)$ is concave for all $t\in T$, and the map $t\mapsto \rho(s,t)$ is convex for all $s\in S$. Then $(S,T,\rho)$ has a value and both players have optimal strategies. 
\item[$(ii)$] \emph{Monotonicity}. Suppose that $(S,T,\rho)$ and $(S,T,\nu)$ have a value, and 
$\rho(s,t)\leq \nu(s,t)$ holds for all $(s,t)\in S\times T$. Then 
$\val\, \rho\leq \val\, \nu$.
\end{itemize}
\paragraph{Matrix games.} In the sequel, we identify every real matrix $M=(m_{a,b})$ of size $p\times q$ with the zero-sum game $(S_M,T_M,\rho_M)$ where $S_M=\De(\{1,\dots,p\})$, $T_M=\De(\{1,\dots,q\})$ and where 
\[
\rho_M(s,t)=\sum_{a=1}^p  \sum_{b=1}^q s(a)\,m_{a,b}\,t(b)\qquad
\forall (s,t)\in S_M\times T_M.
\]
The value of the matrix $M$, denoted by $\val\, M$, is the value of $(S_M,T_M,\rho_M)$ which exists by the minmax theorem. The following properties are well-known:
\begin{itemize}
\item[$(iii)$] \emph{Continuity.} Suppose that $M(t)$ is a
matrix with entries that depend continuously on some
parameter $t\in \R$. Then the map
$t\mapsto \val\, M(t)$ is continuous. 
\item[$(iv)$] \emph{A formula for the value.} For any matrix
$M$, there exists a square sub-matrix $\hat{M}$ of $M$ so that
$\val\,M=\frac{\det \hat{M}}{\varphi(\hat{M})}$, where $\varphi(\hat{M})$ denotes the sum of all the co-factors of $\hat{M}$, with the convention that $\varphi(\hat{M})=1$ if $\hat{M}$ is of size $1\times 1$.
\end{itemize}

\paragraph{Comments.} Property $(i)$ is taken from Sion
\cite{sion58}, a generalization of von Neumann's \cite{VN28}
minmax theorem, while Property $(iv)$ was established by
Shapley and Snow \cite{SS50}. The other two properties are
straightforward. 

\subsection{Stochastic games}\label{NF2}
We present now the standard model of \textbf{finite} competitive stochastic games, henceforth stochastic games for simplicity.  We refer the reader to Sorin's book \cite[Chapter 5]{sorin02} and to Renault's notes \cite{renaultnotes2} for a more detailed presentation of stochastic games.

\paragraph{Definition.} 
A \emph{stochastic game} is described by a tuple $(K,I,J,g,q,k)$, where
$K=\{1,\dots,n\}$ is a finite set of states, for some $n\in \N$, $I$ and $J$ are the finite action sets of Player 1 and 2, respectively, $g:K\times I\times J\to \R$ is a reward function to Player~1, $q:K\times I\times J\to \De(K)$ is a transition function, and 
$1\leq k\leq n$ is an initial state.

The game proceeds in stages as follows. At each stage $m\geq 1$, both players are informed of the current state $k_m\in K$, where $k_1=k$.  Then, independently and simultaneously,  Player 1 chooses an action $i_m\in I$ and Player 2 chooses an action $j_m\in J$. The pair $(i_m,j_m)$ is then observed by both players, from which they can infer the stage reward 
$g(k_m,i_m,j_m)$. A new state $k_{m+1}$ is then chosen
according to the probability distribution $q(k_m,i_m,j_m)$, and the game proceeds to stage $m+1$. 

\paragraph{Discounted stochastic games.} For any discount rate $\la \in(0,1]$, we denote by 
$(K,I,J,g,q,k,\la)$ the stochastic game $(K,I,J,g,q,k)$ where Player 1 maximizes, in expectation, the normalized $\la$-discounted sum of rewards 
\[
\sum\nolimits_{m\geq 1}\la(1-\la)^{m-1} g(k_m,i_m,j_m),
\]
while Player 2 minimizes this amount.\\

In the following, the discount rate $\la$ and the initial
state $k$ will be considered as parameters, while
$(K,I,J,g,q)$ is fixed. 

\paragraph{Strategies.}
A behavioral {strategy}, henceforth a \emph{strategy}, is a
decision rule from the set of possible observations of a
player to the set of probabilities over the set of his actions.
Formally, a strategy for Player 1 is a sequence of mappings
$\si=(\si_m)_{m\geq 1}$, where $\si_m: (K\times I\times
J)^{m-1}\times K\to \De(I)$. 
Similarly, a strategy for Player 2 is a sequence of mappings $\tau=(\tau_m)_{m\geq 1}$, where $\tau_m: (K\times I\times J)^{m-1}\times K\to \De(J)$. 
The sets of strategies are denoted, respectively, by $\Sigma$ and~$\mathcal{T}$. 

\paragraph{The expected pay-off.} By the Kolmogorov extension theorem, together with an initial
state $k$ and the transition function $q$, any pair of strategies $(\si,\tau)\in
\Sigma\times \mathcal{T}$ induces a unique probability 
$\prob_{\si,\tau}^k$
over the sets of plays $(K\times I\times J)^{\N}$ on the
sigma-algebra generated by the cylinders. 
Hence, to any pair of strategies
$(\si,\tau)\in \Sigma\times \mathcal{T}$ corresponds a
unique pay-off $\ga_\la^k(\si,\tau)$ in the discounted game $(K,I,J, g,q,k,\la)$,
\[
\ga_\la^k(\si,\tau):=
\E_{\si,\tau}^k \left[\sum\nolimits_{m\geq 1}\la(1-\la)^{m-1}
g(k_m,i_m,j_m)\right]
\]
where $\E_{\si,\tau}^k$ denotes the expectation with respect
to the probability $\prob_{\si,\tau}^k$.

\paragraph{Stationary strategies.} A \emph{stationary strategy} is a strategy that depends only on the current state. 
Thus, $x:K\to \De(I)$ is a stationary strategy for Player $1$ while $y:K\to \De(J)$ is a stationary strategy for Player $2$. The sets of stationary strategies are $\De(I)^n$ and $\De(J)^n$, respectively. 
A \emph{pure stationary strategy} is a stationary strategy that is deterministic. The sets of pure stationary strategies are $I^n$ and $J^n$, respectively, and we will refer to pure stationary strategies with the bold signs $\ii\in I^n$ and $\jj\in J^n$.
\paragraph{A useful expression.} 

Suppose that both players use stationary strategies $x$ and $y$ in the discounted stochastic game $(K,I,J,g,q,k,\la)$ for some $\la\in(0,1]$. The evolution of the state then follows a Markov chain, and the stage rewards depend only on the current state. Let $Q(x,y)\in \R^{n\times n}$ and $g(x,y)\in \R^{n}$ denote, respectively, the corresponding transition matrix and the vector of expected rewards. Formally, for all $1\leq \ell,\ell'\leq n$, 
\begin{eqnarray}\label{Qg1} 
Q^{\ell,\ell'}(x,y)&=&\sum_{(i,j)\in I\times J} x^\ell(i)y^\ell(j)q(\ell'\,|\, \ell,i,j)\\ \label{Qg2} 
g^\ell(x,y)&=& \sum_{(i,j)\in I\times J} x^\ell(i)y^\ell(j) g(\ell,i,j)\,.
\end{eqnarray}
Let 
$\ga_\la(x,y)=(\ga_\la^1(x,y),\dots,\ga_\la^n(x,y))\in \R^n$.  
Then $Q(x,y)$, $g(x,y)$ and $\ga_\la(x,y)$ satisfy the relations
\begin{eqnarray*}
\ga_\la(x,y)
&=&\sum\nolimits_{m\geq 1}\la (1-\la)^{m-1} Q^{m-1}(x,y) g(x,y)
\\
&=&\la g(x,y)+(1-\la)Q(x,y)\ga_\la(x,y).
\end{eqnarray*}
Let $\Id$ denote the identity matrix of size $n$.  The matrix $\Id-(1-\la)Q(x,y)$ is invertible, as $Q(x,y)$ is a stochastic matrix and $\la\in(0,1]$. Consequently, 
$\ga_\la(x,y)=\la (\Id-{(1-\la)}Q(x,y))^{-1} g(x,y)$. Thus, by Cramer's rule, 
\begin{equation}\label{qut2} \ga^k_\la(x,y)=\frac{d^k_\la(x,y)} {d^0_\la(x,y)}, \end{equation}
where  $d^0_\la(x,y)=\det(\Id - (1-\la)Q(x,y))$ and where $d^k_\la(x,y)$ is the determinant of the $n\times n$-matrix obtained by replacing the $k$-th column of $\Id - (1-\la)Q(x,y)$ with $\la g(x,y)$.

\paragraph{The discounted values.} 
The discounted stochastic game $(K,I,J,g,q,k,\la)$ and the zero-sum game 
$(\Sigma, \mathcal{T},\ga_\la^k)$ are equal by construction.  Thus, the discounted stochastic game has a value whenever
$$\sup_{\si \in \Sigma} \inf_{\tau\in \mathcal T} \ga_\la^k(\si,\tau)=\inf_{\tau \in \mathcal T} \sup_{\si\in 
 \Sigma} \ga_\la^k(\si,\tau).$$
In this case, the value is denoted by $v_\la^k$, and is often referred to as the $\la$-discounted value of the stochastic game $(K,I,J,g,q,k)$. 
The following result is due to Shapley \cite{shapley53}:
 \begin{itemize}
 \item[$(v)$] Every discounted stochastic game $(K,I,J,g,q,k,\la)$ has a value, 
  and both players have optimal stationary strategies. For each $1\leq \ell\leq n$ and $u\in \R^n$, consider the following matrix of size $|I|\times |J|$: 
\begin{equation*}\label{sgame}\mathcal{G}^\ell_{\la,u}:=\left(\la
g(\ell,i,j)+(1-\la)\sum\nolimits_{\ell'=1}^n q(\ell'\mid\ell,i,j)u^{\ell'}\right)_{i,j}\,.\end{equation*}
The vector of values $v_\la=(v^1_\la,\dots,v^n_\la)$ 
is then the unique fixed point of the Shapley operator $\Phi(\la,\,\cdot\,):\R^n\to \R^n$, which is defined by
$\Phi^\ell(\la,u):= \val \, \mathcal{G}^\ell_{\la,u}$, for all $1\leq \ell\leq n$ and $u\in\R^n$. 
\end{itemize}

\paragraph{Remark.}
In the model of stochastic games, 
the discount rate stands for the degree of impatience of the players, in the sense that future rewards are discounted. 
Alternatively, one can interpret $\la$ as the probability that the game stops after every stage. 
The more general case of stopping probabilities that depend on the current state and on the players' actions can be handled in a similar way, as already noted by Shapley \cite{shapley53}.

\paragraph{The value.} The stochastic game $(K, I, J, g, q, k)$ has a value if there exists $v^k\in \R$ such that for any $\ep>0$ there exists $M_0$ such that Player 1 can guarantee that for any $M_0\leq M\leq +\infty$ the expectation of the average
reward per stage in the first $M$ stages of the game is at least $v^k -\ep$, and Player 2 can guarantee that
this amount is at most $v^k-\ep$. It follows that if the game has a value $v^k$, then for each $\ep>0$
there exists a pair of strategies $(\si_\ep,\tau_\ep)\in \Sigma\times  \mathcal{T}$ such that, for some $\la_0\in (0,1]$, the following inequalities hold for all 
$\la\in(0,\la_0)$:
\begin{eqnarray*}
\ga_\la^k(\si_\ep, \tau)&\geq& v^k-\ep\qquad \forall \tau \in \mathcal{T}\\
\ga_\la^k(\si, \tau_\ep)&\leq& v^k+\ep\qquad \forall \si \in \Sigma\,.\end{eqnarray*}
The following result is due to Mertens and Neyman \cite{MN81}:
 \begin{itemize}
 \item[$(vi)$] Every stochastic game $(K,I,J,g,q,k)$ has a value $v^k$, and $v^k=\lim_{\la\to 0} v^k_\la$.
 \end{itemize}

\subsection{State of the art} \label{stateart}
Since its introduction by Shapley \cite{shapley53}, the theory of stochastic games and its applications have been studied in several scientific disciplines.
We restrict our brief literature survey to the theory of
finite competitive stochastic games and related algorithms.

\paragraph{The discounted values.}
In 1953, Shapley \cite{shapley53} proved that every discounted
stochastic game $(K,I,J,g,q,k,\la)$ admits a value $v^k_\la$, 
and that both players have optimal stationary strategies.
Furthermore, the vector of values
$v_\la=(v^1_\la,\dots,v^n_\la)$ is the 
unique fixed point of an explicit operator. 

\paragraph{Existence of the value.} Building on Shapley's
characterization of the discounted values and on a deep
result from real algebraic geometry, the so-called
Tarski-Seidenberg elimination theorem, Bewley and Kohlberg
\cite{BK76} proved in 1976 that the discounted
values converge as the discount rate tends to zero.
Mertens and Neyman \cite{MN81,MN82} strengthened this result in
the early 1980s by establishing that every stochastic game $(K,I,J,g,q,k)$ has a value $v^k$, and that the value coincides with the limit of the
discounted values. It is worth noting that, unlike discounted stochastic games, where
the observation of the past actions is irrelevant, the
existence of the value relies on the observation of the stage 
rewards. 

\paragraph{Alternative proofs of convergence.}
In the late 1990s, 
Szczechla, Connell, Filar, and Vrieze \cite{SCFV97}
gave an alternative proof for the convergence of the
discounted values as the discount rate goes to zero, using Shapley's characterization of the discounted
values and the geometry of complex analytic varieties.
Another proof was recently obtained by Oliu-Barton \cite{OB14}, based on the
theory of finite Markov chains and on Motzkin's alternative theorem for linear systems. 

\paragraph{Robustness of the value.} 
The years 2010s have brought many new results concerning the value of stochastic games. 
Neyman and Sorin
\cite{NS10} studied stochastic games with a random duration
clock. That is, at each stage, the players receive an
additional signal which carries information about the number
of remaining stages. Assuming that the expected number of
remaining stages decreases throughout the game, and that the
expected number of stages converges to infinity, the values
of the stochastic games with a random duration clock
converge, and the limit is equal to the value of the
stochastic game. 
Ziliotto \cite{ziliotto16} considered weighted-average
stochastic games, that is, stochastic games where Player 1
maximizes in expectation a fixed weighted average of the sequence of
rewards, namely $\sum_{m\geq 1}\theta_m g(k_m,i_m,j_m)$.
If $\sum_{m\geq 1}|\theta_{m+1}^p-\theta_m^p|$ converges to zero for some
$p>0$, then the values of the weighted-average stochastic games
converge, and the limit is equal to the value of the
stochastic game.
Neyman \cite{neyman13} considered
discounted stochastic games in continuous time and proved
that their value coincides with the value of the discrete
model.
Finally, Oliu-Barton and Ziliotto \cite{OBZ18} proved
that stochastic games satisfy the \emph{constant pay-off}
property, as conjectured by Sorin, Venel and Vigeral
\cite{SVV10}. That is, for sufficiently small $\la$, any pair of optimal strategies of the discounted game $(K,I,J,g,q,k,\la)$ 
has the property that, in expectation, the average of the cumulated $\la$-discounted sum of rewards on any set of consecutive stages of cardinality of order $1/\la$ is approximately equal to $v^k$. 

\paragraph{Characterization of the value.} 
The first results on the value of stochastic games go back to the mid 1960s. By adapting the tools developed by Howard \cite{howard60} for Markov decision problems, Hoffman and Karp \cite{HK66} obtained a characterization for the limit of the $\la$-discounted values in the irreducible case (that is, when any pair of stationary strategies induces an irreducible Markov chain), in the spirit of an average cost optimality equation. 
Soon after,  Blackwell and Ferguson \cite{BF68} determined the value of the ``Big Match'', an example of a stochastic game whose value depends on the initial state. 
In the mid 1970s, Kolhberg \cite{kohlberg74} introduced \emph{absorbing games}, a class of stochastic games in which there is at most one transition between states, and which 
includes the Big Match as a particular case.
Kohlberg proved that these games have a value, and provided a characterization using the derivative of Shapley's operator. Two additional characterizations for the value 
of absorbing games were obtained recently by Laraki \cite{laraki10} and by Sorin and Vigeral \cite{SV13}, respectively. 

\paragraph{Algorithms.} Whether the value of a finite stochastic game can be computed in polynomial time is a famous open problem in computer science. This problem is intriguing because the class of simple stochastic games is both NP and co-NP, and several important problems with this property have eventually been shown to be polynomial-time solvable, such as primality testing or linear programming. (A simple stochastic game is one where the transition function depends on one player's action at each state.)
The known algorithms fall into two categories: 
decision procedures for the first order theory of the reals, such as 
\cite{chatterjee08, etessami2006, SV10}, and value or
strategy iteration methods, such as \cite{chatterjee2006,rao73}. All of them are worst-case exponential in the number of states or in the number of actions.
Recently, Hansen, Kouck\'y, Lauritzen, Miltersen and
Tsigaridas \cite{HKMT11} achieved a remarkable improvement
by providing an algorithm which is polynomial in the number
of actions, for any {fixed number of states}. However,
the dependence on the number of states is both non-explicit and
doubly exponential. 
 Based on the characterization of the
value obtained in the present paper, Oliu-Barton
\cite{oliubarton2018new2} improved the algorithm of Hansen et al. \cite{HKMT11} by
significantly reducing the dependence on the number of
states to an explicit polynomial dependence on the number of
pure stationary strategies. Although not polynomial in the
number of states, this algorithm is the most efficient
algorithm that is known today.

\subsection{Main results}\label{main}
As already argued, the value is a very robust solution
concept for stochastic games. Its existence was proved
nearly 40 years ago, and an explicit characterization has been missing
since then. The main contribution of the present paper is to
provide a tractable formula for the value of stochastic
games. 

Our result relies on a new characterization of the
discounted values, which is obtained by reducing a
discounted stochastic game with $n$ states to $n$
independent parameterized matrix games, one for each initial
state. 

For the rest of the paper, 
$1\leq k\leq n$ denotes a fixed initial state.
The parameterized game that corresponds to $k$ is simply obtained by linearizing the ratio in \eqref{qut2} for all pairs of pure stationary strategies, as follows. 

\begin{definition}\label{defW} For any $z\in \R$, define the matrix $W^k_\la(z)$ of size $|I|^n\times |J|^n$ 
by setting
\begin{equation*}
W^k_\la(z)[\ii,\jj]:=d^k_\la(\ii,\jj) - z d^0_\la(\ii,\jj)
\qquad \forall (\ii,\jj)\in I^n\times J^n.
\end{equation*}  
\end{definition}

\paragraph{Theorem 1 (A formula for the discounted values).}
\emph{For any $\la\in(0,1]$, the value of the
 discounted stochastic game $(K,I,J,g,q,k,\la)$ is the
unique solution to 
\[z\in \R,\quad \val \, W_\la^k(z) =0\,.\]}

\paragraph{Theorem 2 (A formula for the value).}\emph{For any $z\in \R$, the limit $F^k(z):=\lim_{\la\to 0}\val~W^k_\la(z)/\la^n$
exists in $\R\cup\{\pm\infty\}$. The value of the stochastic game $(K,I,J,g,q,k)$ is 
the unique solution to 
\[
w\in \R,\quad \begin{cases} z> w &\ \Rightarrow\quad   F^k(z)<0 \\
z< w &\ \Rightarrow \quad  F^k(z)>0\,. \end{cases}\]}

\paragraph{Comments}\begin{enumerate}
\item {Theorem 1} provides an \emph{uncoupled}
characterization of the discounted values. That is, each
initial state is considered separately. This property, which
contrasts with Shapley's \cite{shapley53} characterization,
provides the key to {Theorem 2}. 
\item {Theorem 1}
can be extended to stochastic games with
compact action spaces and continuous pay-off and transition
functions, but {Theorem~2} cannot because the discounted values may fail to
converge in this case.
\item {Theorem 2}
provides a new and elementary proof of the convergence of the $\la$-discounted values as $\la$ tends to $0$. 
\item {Theorem 2} captures the characterization of the value for absorbing games obtained by Kohlberg \cite{kohlberg74}.
\item The sign of $F^k(z)$ can be easily computed using
linear programming techniques. This is a crucial aspect of
the formula of {Theorem 2}. 
\item {Theorems 1} and {2}
suggest binary search algorithms for computing, respectively, the discounted values and the value, by successively evaluating the sign of $\val \, W^k_\la(z)$ and of
 $F^k(z)$ for well-chosen $z$. These algorithms are
 polynomial in the number of 
 pure stationary strategies. The precise description and
 analysis of these algorithms is the object of a separate
 paper \cite{oliubarton2018new2}. For completeness, we provide a brief
 description in Section \ref{algos}. 
\end{enumerate} 

\section{A formula for the discounted values}\label{v_la}
In this section we prove {Theorem 1}.
In the sequel, we consider a fixed discounted stochastic game $(K,I,J,g,q,k,\la)$. 
The proof is based on the following four properties: \\[0.15cm]
$1.$ \, $d^0_\la(\ii,\jj)$ is positive for all $(\ii,\jj)\in I^n\times J^n$.\\[0.1cm]
$2.$ \, $(x,y,z)\mapsto d^0_\la(x,y)-z d^k_\la(x,y)$ is a multi-linear map. \\[0.1cm]
$3.$ \, $z\mapsto \val\, W^k_\la(z)$ is a strictly decreasing real map. \\[0.1cm]
$4.$ \, $\val~W^k_\la(v^k_\la) =0$. \\[-0.1cm]

Indeed, {Theorem 1}
clearly follows from the last two. 
The extension of this result to the more general
framework of  compact-continuous stochastic games (that is,
stochastic games with compact metric action spaces and
continuous pay-off and transition functions) proceeds
along the same lines, and is postponed to Section~\ref{cc}.
 
\paragraph{Notation.} We use the following notation:
\begin{itemize}
\item
For any $x=(x^1,\dots,x^n)\in \De(I)^n$ we denote by $\hat{x}\in \De(I^n)$ the element that corresponds to the
direct product of the coordinates of $x$. Formally, 
\[
\hat{x}(\ii):= \prod_{\ell=1}^n x^\ell(\ii^\ell)\qquad \forall\, \ii=(\ii^1,\dots,\ii^n)\in I^n.
\]
The map $x\mapsto \hat{x}$ is one-to-one, and defines the canonical inclusion $\De(I)^n\subset \De(I^n)$. 
The map $y\mapsto \hat{y}$ is defined similarly, and gives the canonical inclusion $\De(J)^n\subset \De(J^n)$.
 \item
The bold letters $\xx$ and $\yy$ refer to elements of
 $\De(I^n)$ and $\De(J^n)$, respectively. 

\item  For all $z\in \R$ and all $(\xx,\yy)\in \De(I^n)\times \De(J^n)$ we set
\[
 W^k_\la(z)[\xx,\yy]:=
 \sum_{(\ii,\jj)\in I^n\times J^n}\xx(\ii)\,W_\la^k(z)[\ii,\jj]\,\yy(\jj).
\]
\end{itemize} 

We now prove the four properties above.
The first is due to Ostrovski \cite{ostrovski37}, and
for completeness we provide a short proof.

 \begin{lemme}\label{l1} For any stochastic matrix $P$ of size $n\times n$ and any $\la\in(0,1]$, $\det(\Id-(1-\la)P)\geq \la^n$. 
 \end{lemme}
 \begin{proof} Set $M:=\Id-(1-\la)P$. Because $P$ is 
a stochastic matrix, $M^{\ell,\ell} - \sum_{\ell'\neq \ell } | M^{\ell, \ell'}|\geq \la$ for all $1\leq \ell\leq n$. Hence, $M$ is strictly diagonally dominant. For any $\mu \in \mathbb{R}$ so that $\mu < \lambda$, the matrix $M - \mu \Id$ is still strictly diagonally dominant, so in particular it is invertible. Consequently, all real eigenvalues of $M$ are larger than or equal to $\la$. Similarly, 
for any $\mu=a+bi \in \mathbb{C}$ so that $|\mu|:=\sqrt{a^2+b^2} < \lambda$, the matrix $M - \mu \Id$ is strictly diagonally dominant, so that $M - \mu \Id$ is invertible. 
Consequently, if $a+b i$ is a complex eigenvalue of $M$,
then $\la\leq |a+bi|$, so that $\la^2\leq |a+bi|^2=a^2+b^2=(a+bi)(a-bi)$.
Recall that $\det M=\prod_{\ell=1}^n \mu_\ell$,
where $\mu_1,\dots,\mu_n$ are the eigenvalues of $M$ counted with multiplicities. Because each real eigenvalue contributes at least $\la$ in the product, and each pair of conjugate eigenvalues contributes at least $\la^2$, it clearly follows that $\det M\geq \la^n$. 
 \end{proof}

\begin{lemme}\label{l2} For any $(x,\jj)\in \De(I)^n\times J^n$ and $z\in \R$,
\begin{itemize}
\item[$(i)$] $d^0_\la(x,\jj)=\sum_{\ii\in I^n} \widehat{x}(\ii) d^0_\la(\ii,\jj)$.
\item[$(ii)$] $d^k_\la(x,\jj)=\sum_{\ii\in I^n} \widehat{x}(\ii) d^k_\la(\ii,\jj)$.
\item[$(iii)$] $W^k_\la(z)[\hat{x},\jj]= d^k_\la(x,\jj)-zd^0_\la(x,\jj)$. 
\end{itemize}
 \end{lemme}
 \begin{proof} $(i)$ Let $\jj\in J^n$ be fixed. For any $x\in \De(I^n)$ set 
$M(x,\jj):=\Id-(1-\la)Q(x,\jj)$, so that $\det M(x,\jj)=d^0_\la(x,\jj)$ and, in particular, $\det M(\ii,\jj)=d^0_\la(x,\jj)$ for all $\ii\in I^n$.
By \eqref{Qg1}, the first row of $M(x,\jj)$ depends on $x$ only through $x^1$, and the dependence is linear. Write $x$ as a convex combination of the stationary strategies $\{(i,x^2,\dots,x^n), \ i\in I\}$, and use the
multi-linearity of the determinant to obtain
\begin{eqnarray*}\det M(x,\jj)&=&\det\left(\sum\nolimits_{i\in I}x^1(i)M((i,x^2,\dots,x^n), \jj)\right)\\
&=& \sum\nolimits_{i \in I}x^1(i)\det
M\bigl((i,x^2,\dots,x^n), \jj\bigr)\,.
\end{eqnarray*}
Using the same argument for the remaining rows, one inductively obtains that 
$\det M(x,\jj)$ is equal to
\[\sum_{\ii^1\in I} x^1(\ii^1)\sum_{\ii^2\in I} x^2(\ii^2)\dots
\sum_{\ii^n\in I} x^n(\ii^n) 
\det M\bigl((\ii^1,\ii^2,\dots,\ii^n), \jj\bigr),\]
which is equal to $\sum_{\ii\in I^n} \widehat{x}(\ii) \det M(\ii, \jj)$ by the definition of $\hat{x}$. \\[0.15cm] 
$(ii)$ The proof goes along the same lines as $(i)$. Fix $\jj\in J^n$. For any $x\in \De(J)^n$, let $M^k(x,\jj)$ be the matrix obtained by replacing the
$k$-th column of $M(x,\jj)$ by $\la g(x,\jj)$, so that $\det M^k(x,\jj)=d^k_\la(x,\jj)$ and, in particular, $\det M^k(\ii,\jj)=d^k_\la(\ii,\jj)$ for all $\ii\in I^n$. By \eqref{Qg1} and \eqref{Qg2}, 
the $\ell$-th row of $M^k(x,\jj)$ depends on $x$ only through $x^\ell$, and that the dependence is linear. Like in $(i)$, these properties imply 
the desired result, namely $\det M^k(x,\jj)=\sum\nolimits_{\ii\in I^n} \widehat{x}(\ii) \det M^k(\ii, \jj)$. \\[0.15cm]  
$(iii)$ The result follows directly from $(i)$, $(ii)$, and the definition of $W^k_\la(z)[\widehat{x},\jj]$. Indeed,
\begin{eqnarray*}
W^k_\la(z)[\widehat{x},\jj]&= &
\sum\nolimits_{\ii\in I^n} \widehat{x}(\ii) W_\la^k(z)[\ii,\jj] 
\\&=& 
\sum\nolimits_{\ii\in I^n} \widehat{x}(\ii)
d^k_\la(\ii,\jj) - z \sum\nolimits_{\ii\in I^n}
\widehat{x}(\ii) d^0_\la(\ii,\jj)\\
&=& d^k_\la(x,\jj) - z  d^0_\la(x,\jj).
\end{eqnarray*} 
\end{proof}

\paragraph{Remark.} Lemma \ref{l2} is stated for all $(x,\jj)$ for convenience, but is also valid for all $(x,y)$. The last property, for instance, can be stated as follows. 
For all $(x,y,z)\in \De(I)^n\times \De(J)^n\times \R$, 
\[W^k_\la(z)[\hat{x},\hat{y}]= d^k_\la(x,y)-zd^0_\la(x,y).\]


\begin{lemme}\label{ineq} For any $(z_1,z_2)\in \R^2$ 
so that $z_1<z_2$, 
\begin{equation*}\label{ineq}\val~W^k_\la(z_1) -\val~W^k_\la(z_2) \geq (z_2-z_1)  \la^n\, .\end{equation*} 
In particular, $z\mapsto \val \, W^k_\la(z)$ is a strictly decreasing real map. 
\end{lemme}
\begin{proof} 
By definition, $Q(\ii,\jj)$ is a stochastic matrix of size $n\times n$ for each $(\ii,\jj)\in I^n\times J^n$. Hence, by Lemma  \ref{l1}, 
\[d^0_\la(\ii,\jj)=\det(\Id-(1-\la)Q(\ii,\jj))\geq \la^n\qquad \forall (\ii,\jj)\in I^n\times J^n\,.\] Therefore, for all $z_1<z_2$ and $(\ii,\jj)$,
\begin{eqnarray*}
W^k_\la(z_1)[\ii,\jj]- W^k_\la(z_2)[\ii, \jj]
 &=&
(z_2-z_1) d^0_\la(\ii,\jj)\\
&\geq &  (z_2-z_1) \la^n\,.
\end{eqnarray*}
The result follows then from the monotonicity of the value operator, stated in item $(ii)$ of Section \ref{NF}. 
 \end{proof}

\begin{lemme}
\label{val=0}
$\val~W^k_\la(v^k_\la)= 0$\,.
\end{lemme}
\begin{proof} 
By Lemma \ref{l2} $(iii)$, the relation \begin{equation}\label{pkp}W^k_\la(v^k_\la)[\widehat{x},\jj]=d^k_\la(x,\jj)- v^k_\la
d^0_\la(x,\jj)\end{equation} holds 
for all $(x,\jj)\in \De(I)^n \times J^n$. 
Let $x^*\in \De(I)^n$ be an optimal stationary
strategy of Player $1$ in $(K,I,J,g,q,k,\la)$, which exists by Shapley \cite{shapley53} as already noted in item $(v)$ of Section~\ref{NF2}, and let 
$\widehat{x}^*\in \De(I^n)$ denote the direct product of its coordinates. The optimality of $x^*$ implies
\begin{equation*}
\ga_\la^k(x^*,\jj)=\frac{d^k_\la(x^*,\jj)}{d^0_\la(x^*,\jj)}\geq
v^k_\la.
\end{equation*}
The matrix $Q(x^*,\jj)$ is stochastic of size $n\times n$ so that  $d^0_\la(x^*,\jj)=\det(\Id-(1-\la)Q(x^*,\jj))\geq \la^n>0$ by Lemma \ref{l1}. Consequently, the previous relation is equivalent to 
 \begin{equation}\label{pkp2}d^k_\la(x^*,\jj)- v^k_\la
d^0_\la(x^*,\jj)\geq 0.\end{equation}
Therefore, $W^k_\la(v^k_\la)[\widehat{x}^*,\jj]\geq 0$ follows from \eqref{pkp} and \eqref{pkp2}. 
For any matrix $M=(m_{a,b})$ of size $p\times q$ and any $s\in \De(\{1,\dots,p\})$, the definition of the value implies that $\val\,  M\geq \min_{1\leq b\leq q} \sum_{1\leq a\leq p} s(a) m_{a,b}$.  
 Consequently, 
\begin{equation*}\label{aa1} 
\val~W^k_\la(v^k_\la) \geq \min_{\jj\in J^n} W^k_\la(v^k_\la)[\widehat{x}^*,\jj]\geq 0\, .
\end{equation*}
By reversing the roles of the players one similarly obtains
an analogue of Lemma \ref{l2} for all $(\ii,y)\in I^n \times \De(J)^n$, and
then $\val~W^k_\la(v^k_\la)\leq 0$, which gives the desired result. 
\end{proof} 


\paragraph{Proof of Theorem 1.} By Lemma \ref{ineq}, 
$z\mapsto \val \, W^k_\la(z)$ is a strictly decreasing real function. 
Consequently, the set
$\{z\in\R,\ \val~W^k_\la(z)=0\}$
contains at most one element.
By Lemma \ref{val=0}, this element is precisely $v^k_\la$. \hfill $\blacksquare$

\section{A formula for the value}\label{the_limit}
In this section we prove {Theorem 2}.
Before we establish this result, we show that 
the limit $F^k(z):=\lim_{\la\to 0} \val\, W^k_\la(z)/ \la^n$ exists in $\R\cup\{-\infty,+\infty\}$ for all $z\in \R$, and that the equation
\begin{equation}\label{ab}
w\in \R,\quad \begin{cases} z> w &\ \Rightarrow\quad   F^k(z)<0 \\
z< w &\ \Rightarrow \quad  F^k(z)>0\, \end{cases}\end{equation} 
admits a unique solution.
This is shown in the following two lemmas.
\begin{lemme}\label{rty}
\label{rational}
Let $z\in \R$.
Then, there exists a rational fraction $R$ and $\la_0>0$ so that
\[
\val\, W^k_\la(z)=R(\la)
\qquad \forall \la\in(0,\la_0)\,.
\]
\end{lemme}
\begin{proof} 
 By construction, 
 the entries of $W^k_\la(z)$ are polynomials in $\la$. By Shapley and Snow \cite{SS50}, the value of a matrix satisfies the formula stated in item $(iv)$ of Section~\ref{NF}. Consequently, for any $\la\in (0,1]$, there exists a rational fraction $R$ 
so that $\val\, W^k_\la(z) =R(\la)$. 
Because the choice of the square sub-matrix may vary with $\la$, the corresponding rational fraction may also vary. However, as the number of possible square sub-matrices is finite, so is the number of possible rational fractions that may satisfy this equality. 
Consequently,  
there exists a finite collection $E=\{R_1,\dots,R_L\}$  of rational fractions
so that
for each $\la\in(0,1]$ there exists $R\in E$ that satisfies
$\val\, W^k_\la(z)=R(\la)$. Hence, for any $\la$, the point $(\la, \val\, W^k_\la(z))$ belongs to the union of the graphs of the functions $R_1,\dots,R_L$. 
As already noted in item $(iii)$ of Section~\ref{NF}, the map $\la\mapsto \val\, W^k_\la(z)$ is continuous on $(0,1]$. Consequently, as $\la$ varies on the interval $(0,1]$, the curve $\la\mapsto (\la, \val\, W^k_\la(z))$ can ``jump'' from the graph of $R$ to the graph of $R'$ only at points where these two graphs intersect. Yet, for any two rational fractions, either they are congruent or they intersect finitely many times. Hence, there exists $\la_0$ so that, for any $R,R'\in E$, either $R(\la)=R'(\la)$ for all $(0,\la_0)$ or $R(\la)\neq R'(\la)$ for all $(0,\la_0)$. 
In particular, there exists $R\in E$ so that $\val\, W^k_\la(z)=R(\la)$ for all $(0,\la_0)$. 
\end{proof}

\begin{lemme} \label{char} Equation \eqref{ab} admits a unique solution. 
\end{lemme} 
\begin{proof} By Lemma \ref{rty}, $\lim_{\la\to 0}\val \, W^k_\la(z)/ \la^n$ exists for all $z\in \R$. 
Suppose that \eqref{ab} admits two solutions $w<w'$. Then, for any $z\in (w,w')$ one has $F^k(z)<0$ and $F^k(z)>0$, which is impossible. Therefore, \eqref{ab} admits at most one solution. 
 Let $(z_1,z_2)\in \R^2$ satisfy $z_1<z_2$.
Rearranging the terms in Lemma \ref{ineq}, dividing by
$\la^n$ and taking $\la$ to $0$ yields
\begin{equation}\label{dec} F^k(z_1)\geq F^k(z_2)+z_2-z_1\, .\end{equation}
In particular, the following relations hold: 
\begin{equation}\label{dych}\begin{cases}
F^k(z)\geq 0 \ \Rightarrow \ F^k(z')\geq 0, \ \forall z'\leq z\\ 
F^k(z)\leq 0 \ \Rightarrow \ F^k(z')\leq 0,\ \forall z'\geq z \\
F^k(z)= 0 \ \Rightarrow \ F^k(z')\neq 0,\ \forall z'\neq z\, .\end{cases}\end{equation}
We now show that $F^k$ is not constant, which is still compatible with \eqref{dec} if $F^k\equiv +\infty$ or 
$F^k\equiv -\infty$.
 Let $C^-:=\min_{k,i,j}g(k,i,j)$ and $C^+:=\max_{k,i,j}g(k,i,j)$. For any $\la\in(0,1]$, one clearly has
$C^-\leq v^k_\la\leq C^+$. 
Consequently, by Lemma \ref{ineq}, 
$$\val \, W^k_\la(C^+)\leq \val \, W^k_\la(v^k_\la)\leq \val \,  W^k_\la(C^-)\, .$$
Dividing by $\la^n$ and taking $\la$ to $0$ one obtains
\begin{equation}\label{dec2} F^k(C^+)\leq 0\leq F^k(C^-)\, .
\end{equation}
We now define recursively two real sequences $(u^-_m)_{m\geq 1}$ and
$(u^+_m)_{m\geq 1}$ by setting $u^-_1:=C^-$, $u^+_1:=C^+$ and, for all $m\geq 1$,
 \[u^-_{m+1}:=\begin{cases}
 \frac{1}{2}(u^-_m+u^+_m) & \text{if }
 F^k\left(\frac{1}{2}(u^-_m+u^+_m)\right)\geq 0\\ 
u^-_m & \text{otherwise},\end{cases}\]
\[u^+_{m+1}:=
\begin{cases}
\frac{1}{2}(u^-_m+u^+_m) & \text{if }
F^k\left(\frac{1}{2}(u^-_m+u^+_m)\right)\leq 0\\ 
u^+_m & \text{otherwise}\, . \end{cases}\]
By construction,
$F^k(u^-_m)\geq 0$ and $F^k(u^+_m)\leq 0$ for all $m\geq 1$. Moreover, \eqref{dych} and \eqref{dec2} imply 
$C^-\leq u^-_m\leq u^+_m\leq C^+$ for all $m\geq 1$, so that $(u^-_m)_m$ is non-decreasing and $(u^+_m)_m$ is non-increasing.
Furthermore, 
$u^+_{m+1}-u^-_{m+1}\leq \frac{1}{2}(u^+_{m}-u^-_{m})$ for all $m\geq 1$. Hence, the two sequences admit a common limit $\bar{u}$. 
For any $\ep>0$, let  $m_\ep$ be such that $u^-_{m_\ep}> \bar{u}-\ep$. By \eqref{dec}, this implies 
$$F^k(\bar{u}-\ep)\geq F^k(u^-_{m_\ep})+ u^-_{m_\ep}-(\bar{u}-\ep)> 0\, .$$ 
Similarly, $F^k(\bar{u}+\ep)<0$ for any $\ep>0$. Together with \eqref{dych}, this shows that $\bar{u}$ is a solution to \eqref{ab}. 
\end{proof}

\bigskip

\noindent We are now ready to prove our main result. 
 \paragraph{Proof of Theorem 2.} Let $w$ be the unique solution \eqref{ab} and fix $\ep>0$. By the choice of $w$, $F^k(w-\ep)>0$. Consequently, there exists 
$\la_0>0$ so that
\begin{equation}\label{eqrd} \val~W^k_\la(w-\ep)>0 
\qquad \forall \la\in(0,\la_0)\, .\end{equation}
By Lemma \ref{ineq}, the map $z\mapsto \val\, W^k_\la(z)$ is strictly decreasing.
By Lemma \ref{val=0}, $\val~W^k_\la(v_\la^k)=0$. 
Therefore, \eqref{eqrd} implies 
\begin{equation}\label{eqrd2}v_\la^k> w-\ep
\qquad \forall \la\in(0,\la_0)\, .\end{equation}
Because $\ep$ is arbitrary, $\liminf_{\la\to 0} v_\la^k\geq w$. 
By reversing the roles of the players,
one obtains in a similar manner $\limsup_{\la\to 0}v_\la^k\leq w$\,.
Hence, the $\la$-discounted values converge as $\la$ goes to $0$, and 
$\lim_{\la\to 0} v_\la^k=w$.
The result follows then from item $(vi)$ of Section~\ref{NF2}, namely the existence of the value $v^k$ and the equality $\lim_{\la\to 0} v_\la^k=v^k$, due to Mertens and Neyman \cite{MN81}.
\hfill $\blacksquare$

\section{Algorithms}\label{algos}
The formulas obtained in
{Theorems 1} and {2}
suggest binary search methods for approximating the $\la$-discounted values and the value of a stochastic game $(K,I,J,g,q,k)$, 
based on the evaluation of the sign of the real functions $z\mapsto \val\, W^k_\la(z)$ and $z\mapsto F^k(z)$, respectively. 
In this section we
provide a brief description of these algorithms, and discuss their complexity using the logarithmic cost model (a model which accounts for the total number of bits which are involved). 
We refer the reader to \cite{oliubarton2018new2} for more technical details, and for two additional algorithms which provide \emph{exact expressions} for $v_\la^k$ and $v^k$ within the same complexity class.


\paragraph{Notation.} For any $m\in \N$, let
$E_m:=\{0,\frac{1}{m},\frac{2}{m},\dots,\frac{m}{m}\}$ and $Z_m:= \{0,\frac{1}{2^m},\frac{2}{2^m},\dots,\frac{2^m}{2^m}\}$.

\subsection{Computing the discounted values}
The following bisection algorithm, which is directly derived from {Theorem 1}, inputs a discounted stochastic game with rational data and outputs an arbitrarily close approximation of its value.\\[-0.15cm] 

\noindent \textbf{Input.} A discounted stochastic game $(K,I,J,g,q,k,\la)$ so
that, for some $(N,L)\in \N^2$, the functions $g$ and $q$ take
values in $E_N$ and $\la\in E_L$, and a precision level $r\in\N$.\\[0.05cm] \noindent
\textbf{Output.} A $2^{-r}$-approximation of
$v^k_\la$.\\[0.1cm] \noindent  \textbf{Complexity.} 
Polynomial in $n$ $|I|^n$, $|J|^n$, $\log N$, $\log L$ and $r$.\\[-0.15cm]

\noindent $1$ \, Set $\underline{w}:=0$, $\overline{w}:=1$\\[0.05cm] 
\noindent $2$ \,  WHILE $\overline{w}-\underline{w}> 2^{-r}$ DO\\
\indent $2.1$ \, $z:= \frac{\underline{w}+\overline{w}}{2}$\\[0.05cm] 
\indent $2.2$ \, $v:=\mathrm{sign}$ of $\val\, W^k_{\la}(z)$\\[0.05cm] 
\indent $2.3$ \, IF $v\geq 0$ THEN $\underline{w}:=z$\\[0.05cm] 
\indent $2.4$ \, IF $v \leq 0$ THEN $\overline{w}:=z$\\[0.05cm] 
\noindent $3$ \, RETURN $u:=\underline{w}$\,. 
\medskip

\noindent By construction, the output $u$ 
satisfies $|u-v_\la^k|\leq 2^{-r}$, and the number of iterations of the ``while'' loop is bounded by $r$. Also, the complexity of each iteration depends crucially on the complexity of Step 2.2. 
First of all, one needs to determine the matrix $W^k_\la(z)$ for some $z\in Z_r$, and this requires the computation of two $n\times n$ determinants for each of its $|I|^n\times |J|^n$ entries. Algorithms for computing the determinant of a matrix exist which are polynomial in its size and in the number of bits that which are needed to encode this matrix. Second, the choice of $z$ and Hadamard's inequality imply that the number of bits which are needed to
encode $W^k_\la(z)$ is 
polynomial in 
$n$, $|I|^n$, $|J|^n$, $\log N$, $\log L$ and~$r$. 
Third, computing the value of a matrix can be done with linear
programming techniques, and algorithms exist (for example, Karmarkar \cite{karmarkar84}) which are
polynomial in its size and in the number of bits
which are needed to encode this matrix. 
Consequently, 
the computation cost of Step~2.2 is 
polynomial in $n$, $|I|^n$, $|J|^n$, $\log N$, $\log L$ and~$r$, and
the same is true for the entire algorithm.

\subsection{Computing the value}
The following bisection algorithm, which is directly derived from {Theorem 2},
 inputs a stochastic game with rational data and outputs an arbitrarily close approximation of its value.\\[-0.15cm] 
 
\noindent \textbf{Input.} A stochastic game $(K,I,J,g,q,k)$ so
that, for some $N\in \N$, the functions $g$ and $q$ take
values in $E_N$, and a precision level $r\in\N$. \\[0.05cm] \noindent
\textbf{Output.} A $2^{-r}$-approximation of
$v^k$.\\[0.1cm] \noindent  \textbf{Complexity.} 
Polynomial in $n$ $|I|^n$, $|J|^n$, $\log N$ and $r$.\\[-0.15cm]

\noindent $1$ \, Set $\underline{w}:=0$, $\overline{w}:=1$\\[0.05cm] 
\noindent $2$ \,  WHILE $\overline{w}-\underline{w}> 2^{-r}$ DO\\
\indent $2.1$ \, $z:= \frac{\underline{w}+\overline{w}}{2}$\\[0.05cm] 
\indent $2.2$ \, $v:=\mathrm{sign}$ of $F^k(z)$\\[0.05cm] 
\indent $2.3$ \, IF $v\geq 0$ THEN $\underline{w}:=z$\\[0.05cm] 
\indent $2.4$ \, IF $v \leq 0$ THEN $\overline{w}:=z$\\[0.05cm] 
\noindent $3$ \, RETURN $u:=\underline{w}$\,. 

\medskip

\noindent Like before, the output $u$ 
satisfies $|u-v^k|\leq 2^{-r}$, the number of iterations of the ``while'' loop is bounded by $r$, and the variable $z$ always
takes values in the set $Z_r$. 
Unlike before, however, each iteration requires computing 
the sign of $F^k(z)$ at Step 2.2, a computation that might seem problematic due to the limiting nature of the function $F^k$.
However, this difficulty can be overcome thanks to the following result.
\begin{proposition}[Proposition 3.6 of \cite{oliubarton2018new2}]\label{propaux} For any $r\in \N$, introduce $\la_r:=4nd(\bit(n)+\bit(d)+\bit(N))-rnd$, where for each $p\in \N$, $\bit(p):=\lceil \log_2(p+1)\rceil$ is the number of bits of $p$. Then, the sign of $F^k(z)$ is equal to the sign of $\val\, W^k_{\la_r}(z)$ for all $z\in Z_r$.\end{proposition}
Indeed, Proposition \ref{propaux} implies that the computation in Step~2.2 can be replaced with the computation of $\val\, W^k_{\la_r}(z)$.
By the choice of $\la_r$ and $z$, the number of bits which are needed to encode $W^k_{\la_r}(z)$ is polynomial in $n$, $|I|^n$, $|J|^n$, $\log N$ and $r$, so that the computation in Step~2.2 is polynomial in these variables, and the same is true for the entire algorithm.


\section{Remarks and extensions} \label{comments}
First, we provide an alternative definition of the parameterized games $W^k_\la(z)$.
Second, we extend
{Theorem 1}
to the more general framework of stochastic games with
compact metric action sets and continuous pay-off and
transition function, and explain why the extension of
{Theorem 2} fails. Finally, we show that the formula
obtained by Kohlberg \cite{kohlberg74} for the value of
absorbing games 
is captured by {Theorem 2}.

\subsection{An alternative formulation of the parameterized games}
The parameterized game $W^k_\la(z)$ plays a crucial role
both in {Theorems 1} and~2.
We provide an alternative construction of this game which is based on the Kronecker product of matrices. 
Let $U$ denote a matrix of ones of size $|I|\times |J|$. 
For each $1\leq \ell,\ell'\leq n$, 
consider the matrices 
$Q^{\ell,\ell'}=(q(\ell'\,|\, \ell,i,j))_{i,j}$ and $G^{\ell}=(g(\ell,i,j))_{i,j}$,  
and use them to form the following $n\times (n+1)$ array of matrices of size $|I|\times |J|$: 
$$D_\la=\begin{pmatrix} -\la G^1 & U-(1-\la)Q^{1,1} & \dots & -(1-\la) Q^{1,n} \\
\vdots & \vdots & \ddots & \vdots \\
-\la G^n & -(1-\la) Q^{n,1} & \dots & U-(1-\la) Q^{n,n}\end{pmatrix}. $$
For any $0\leq \ell\leq n$, let $D_\la^\ell$ be the $n\times n$ array of matrices obtained by removing the $(\ell+1)$-th column of matrices from $D$. Denote by $\det\nolimits_\otimes$ the determinant of a square array of matrices, 
developed along columns and where the products are replaced with the Kronecker product of matrices. 
By construction, $\det\nolimits_\otimes D_\la^0=(d^0_\la(\ii,\jj))_{\ii,\jj}$ and $(-1)^k \det\nolimits_\otimes D_\la^k=(d^k_\la(\ii,\jj))_{\ii,\jj}$,
so that 
\[W^k_\la(z)=(-1)^k \det\nolimits_\otimes D_\la^k - z \det\nolimits_\otimes D_\la^0\,.\]
The linearity relations established in Lemma \ref{l2} can also be deduced from the properties of the Kronecker product.
This alternative expression for $W_\la^k(z)$ is reminiscent of (or, rather, inspired by) the theory of multi-parameter eigenvalue problems initiated by Atkinson in the 1960s, see Chapter 6 of \cite{atkinson72}. 
The interesting connection which exists between stochastic games and multi-parameter eigenvalue problems is developed by L.A. and M.O-B. in a forthcoming paper~\cite{AOB18b2}.


\subsection{Compact-continuous stochastic games}\label{cc}
Throughout this section we consider stochastic games $(K,I,J,g,q)$, where $K=\{1,\dots,n\}$ is a finite set of states, $I$ and $J$ are two compact metric sets, and $g$ and $q$ are continuous functions. 
These games are referred to as compact-continuous stochastic games, for short. 
We denote by $\De(I)$ and $\De(J)$, respectively,
the sets of probability distributions over $I$ and $J$. These sets are compact when endowed with the
weak* topology. For any $(\al,\be)\in  \De(I)\times \De(J)$, we denote its direct product by $\al\otimes \be\in \De(I\times J)$. For all $1\leq \ell,\ell'\leq n$ and $u\in \R^n$, we set 
\begin{eqnarray*}g(\ell,\al,\be)&:=& \int_{I\times J} g(\ell,i,j)\,d(\al \otimes \be)(i,j)\\ q(\ell'|\ell,\al,\be)&:=& \int_{I\times J} q(\ell'|\ell,i,j)\,d(\al \otimes \be)(i,j),\\ 
\rho^\ell_{\la,u}(\al,\be)&:=&
\la g(\ell,\al,\be)+ (1-\la)\sum_{\ell'=1}^n 
q(\ell'|\ell,\al,\be).
\end{eqnarray*}
By the minmax theorem stated in item $(i)$ of
Section~\ref{NF}, the zero-sum game $(\De(I),\De(J),\rho^\ell_{\la,u})$ has a value, so one can define the Shapley operator  $\Phi(\la,\,\cdot\,):\R^n\to \R^n$ like in the finite case. 
Furthermore, the compact-continuous stochastic game $(K,I,J,g,q,k,\la)$ has a value $v_\la^k$, which is the unique fixed point of $\Phi(\la,\,\cdot\,)$, and both players have optimal stationary strategies. 
These results are well-known.

\noindent
\paragraph{Extension of Theorem 1.}
{Theorem 1} can be extended to compact-continuous stochastic
games. 

\medskip

The proof goes along the same lines. 
Like in the finite case, any pair of
stationary strategies $(x,y)\in \De(I)^n\times \De(J)^n$
induces a Markov chain with state-dependent rewards. Let $Q(x,y)\in \R^{n\times n}$ and $g(x,y)\in \R^n$
denote the transition matrix of this chain and the vector of
expected rewards.
Formally, they are defined like in \eqref{Qg1} and
\eqref{Qg2}, but replacing, for $1\leq \ell ,\ell'\leq n$,
 the sum $\sum_{(i,j)\in I\times J} x^\ell(i)y^\ell(j)$ with the corresponding integral $\int_{I\times J} d(x^\ell\otimes y^\ell)(i,j)$. 
Similarly, let $\ga_\la(x,y)\in \R^n$ be the vector of
expected normalized $\la$-discounted sum of rewards, which is
well-defined because the state $k$, the pair $(x,y)$, and
the transition function $q$ induce a unique
probability measure over $(K\times I\times J)^{\N}$ on the
sigma-algebra generated by the cylinders, by the Kolmogorov extension theorem.
Like in the finite case, 
 \begin{equation*}\ga_\la^k(x,y)=\frac{d^k_\la(x,y)} {d^0_\la(x,y)},\end{equation*} 
where $d^0_\la(x,y):=\det(\Id - (1-\la)Q(x,y))\neq 0$ and
where $d^k_\la(x,y)$ is the determinant of the $n\times
n$-matrix obtained by replacing the $k$-th column of $\Id -
(1-\la)Q(x,y)$ with $\la g(x,y)$. Lemma 
 \ref{l2} can be extended word for word, by replacing sums
with the corresponding integrals, and setting
$\widehat{x}:=x^1\otimes \dots\otimes x^n\in \De(I^n)$. 

For each $z\in \R$, the auxiliary game $W^k_\la(z)$ can be defined in a similar manner by setting
$$W^k_\la(z)[\ii,\jj]:=d_\la^k(\ii,\jj)-zd_\la^0(\ii,\jj)\qquad
\forall (\ii,\jj)\in I^n\times J^n.$$ 
Note that $W^k_\la(z)$ is no longer a matrix, but a mapping from the compact metric set $I^n\times J^n$ to $\R$. Like in the finite case, consider the \emph{mixed extension} of this game, that is: 
the zero-sum game with action sets $\De(I^n)$ and $\De(J^n)$ and pay-off function
\[W^k_\la(z)[\xx,\yy]:=\int_{I^n\times J^n}W_\la^k(z)[\ii,\jj]\,
d(\xx\otimes \yy)(\ii,\jj)\,.\] 
By the minmax theorem stated in item $(i)$ of Section~\ref{NF}, this game admits a value, denoted by $\val\, W^k_\la(z)$. 
Lemmas \ref{ineq} and \ref{val=0} can thus be extended
word for word as well; it is enough to replace all sums with the
corresponding integrals.
The extension of {Theorem~1} follows directly from these two lemmas. 

\noindent
\paragraph{Extension of Theorem 2.} 
{Theorem 2} cannot be extended to compact-continuous stochastic
games. 

\medskip

Indeed, Vigeral \cite{vigeral13} provided an example of a
stochastic games with compact action sets and continuous
pay-off and transition functions for which the discounted
values do not converge. In this sense, the extension of our
result to this framework is not possible.
However, we point out that only one point in our proof
is problematic. Indeed, the failure occurs in the use of Lemma \ref{rational}, which relies on
the formula stated as Property $(iv)$ in Section~\ref{NF},
which only holds in the finite case. For infinite action sets it is no longer true that
$\la\mapsto \val~W^k_\la(z)$
is a rational fraction in $\la$ in a neighborhood of $0$ for
all $z\in \R$, which was crucial to prove the existence of the limit
$F^k(z):=\lim_{\la\to 0}\val~W^k_\la(z)/\la^n$.

Determining necessary and sufficient conditions on $I$, $J$,
$g$, and $q$ which ensure the convergence of the discounted
values or the existence of the value is an {open problem}. 
Bolte, Gaubert and Vigeral \cite{BGV14} provided sufficient
conditions, namely that $g$ and $q$ are {separable} and
{definable}. Without going into a precise definition of
these two conditions, they hold in particular when the
pay-off function $g$ and the transition $q$ are polynomials
in the players' actions. However, the case where $I$, $J$,
$g$, and $q$ are {semi-algebraic} is still unsolved.
(A subset $E$ of $\R^d$ is semi-algebraic if it is
defined by finitely many polynomial inequalities; a function
is semi-algebraic if its graph is semi-algebraic.)

\subsection{Absorbing games}
We now show that Kohlberg's 
result \cite{kohlberg74} on absorbing games is captured in 
{Theorem~2}. An \emph{absorbing game} is a 
stochastic game $(K,I,J,g,q,k)$ so that, for some fixed state $k_0\in K$, 
\[
q(k\,|\, k, i,j)=1\qquad \forall (i,j)\in I\times J\,,
\quad \forall \,k\neq k_0\,.
\]
For any initial state $k\neq k_0$, the state does not evolve
during the game and, as a consequence, $v^k_\la$ is equal to the value of the matrix $(g(k,i,j))_{(i,j)\in I\times J}$ for all $\la\in(0,1]$ and $k\neq k_0$. We will use the notation $v^k$ to emphasize that $v_\la^k$ does not depend on $\la$, for all $k\neq k_0$. 

\paragraph{Notation.} We assume without loss of generality that $k_0=1$, and set $u(z):=(z,v^2,\dots,v^n)$ for all $z\in \R$.  
\paragraph{Kolhberg's result.} 
Every absorbing game $(K,I,J,g,q,1)$ has a value, denoted by $v^1$, which is 
the unique point where the function  $T:\R\to  \R\cup\{\pm\infty\}$ changes sign, where $T$ is defined using the Shapley operator by $$T(z):=\lim_{\la\to 0} \frac{\Phi^1(\la,u(z))-z}{\la}\qquad \forall z\in \R\,.$$

\paragraph{Comparison to our result.}
We claim that $F^1=T$ in the class of absorbing games. 
To see this, first of all note that for all $(\ii,\jj)\in I^n\times J^n$,
 \begin{eqnarray*}
d^0_\la(\ii,\jj)&=&\la^{n-1}\left(1-(1-\la)q(1\,|\, 1,\ii^1,\jj^1)\right)\\
d^1_\la(\ii, \jj) 
&=&\la^{n-1}\left(\la g(1,\ii^1,\jj^1)+(1-\la)\sum_{\ell=2}^n q(\ell\,|\, 1,\ii^1,\jj^1)v^\ell \right).\end{eqnarray*}
Thus, for any $z\in \R$, 
\[W^1_\la(z)=\left(\la^{n-1}\left(\la g(1,\ii^1,\jj^1) 
 +(1-\la)\sum\nolimits_{\ell=1}^n q(\ell\,|1,\ii^1,\jj^1)u^\ell(z) - z\right)\right)_{\ii,\jj}\,.\]
In particular, $W^1_\la(z)$ depends on $(\ii,\jj)\in I^n\times J^n$ only through $(\ii^1,\jj^1)\in I\times J$.
Eliminating the redundant rows and columns of $W^1_\la(z)$ one thus obtains the matrix $\la^{n-1}(\mathcal{G}^1_{\la,u}-z U)$, where $\mathcal{G}^1_{\la,u}$ is the  $|I|\times |J|$-matrix described in item $(v)$ of Section~\ref{NF2}, and 
$U$ is a matrix of ones of the same size. 
The affine invariance of the value operator, namely $\val(c M+dU)=c~\val\, M+d$ for any matrix $M$ and any $(c,d)\in (0,+\infty)\times \R$, gives then
\[ \frac{\val\, W^1_\la(z)}{\la^n}=\frac{\la^{n-1}\val\,
(\mathcal{G}^1_{\la,u(z)}-zU)}{\la^n}=\frac{\Phi^1(\la,u(z))-z}{\la}\,.\]  
Taking $\la$ to $0$ gives the desired equality, $F^1=T$. 



\section*{Acknowledgements} We are greatly indebted to Sylvain Sorin, whose comments on an earlier draft led to significant simplifications of our main proofs. We are also very thankful to 
Abraham Neyman for his careful reading and numerous remarks
on a previous version of this paper, and to Bernhard von
Stengel, the Editor, and the anonymous reviewers for their insightful comments and suggestions at a later stage. 
The second author gratefully acknowledges the support of the French National Research Agency
for the Project CIGNE (Communication and Information in Games on Networks) ANR-15-CE38-0007-01, and the support of the Cowles Foundation at Yale University.


\bibliographystyle{plain}
\bibliography{bibliothese2}


\end{document}